\documentclass[a4paper, reqno, 11pt]{article}
\usepackage{mathtools,amssymb,amsthm,mathrsfs,calc,graphicx,xcolor,cleveref,dsfont,tikz,bm}
\usepackage{amssymb}
\usepackage{dsfont}
\newcommand{\mathbbm}[1]{\text{\usefont{U}{bbm}{m}{n}#1}}
\usepackage{tikz}

\DeclareMathOperator*{\argmin}{arg\,min}

\newtheorem{theorem}{Theorem}[section]
\newtheorem{lemma}[theorem]{Lemma}

\newtheorem{corollary}[theorem] {Corollary}

\newtheorem{remark}[theorem] {Remark}

\def\phi{\varphi }

\newcommand{\R}     {\mathbb{R}}

\newcommand{\N}     {\mathbb{N}}
\renewcommand{\P}   {\mathbb{P}}
\newcommand{\E}     {\mathbb{E}}

\renewcommand{\d}   {\operatorname{d}\!}
\newcommand{\diam}   {\operatorname{diam}}

\allowdisplaybreaks[1] 

\usepackage{color}

\begin{document}
\title{Sedentary Random Waypoint}
\author{Carina Betken\footnotemark[1],\; Hanna D\"oring\footnotemark[2]}

\footnotetext[1]{Ruhr-Universit\"at Bochum, Fakult\"at f\"ur Mathematik, D-44801 Bochum, Germany, {\tt carina.betken@ruhr-uni-bochum.de}.}

\footnotetext[2]{Universit\"at Osnabr\"uck, Institut f\"ur Mathematik, D-49076 Osnabr\"uck, Germany, {\tt hanna.doering@uni-osnabrueck.de}.}

\maketitle
\begin{abstract}
	\noindent We adjust the classical random waypoint mobility model used in the study of telecommunication networks to a more realistic setting by allowing participants of the network to return to popular places and individual homes. We show that the two fundamental random times of detection and coverage in this new probabilistic model for large random networks exhibit exponential tails. Furthermore we examine the model for percolation.
\end{abstract}


\section{Introduction}
Spatial random networks usually consist of a random or deterministic number of nodes, equipped with a \textit{communication radius r} allowing them to communicate with the surrounding environment. So far many works have looked at such networks at fixed points in time, see for example \cite{Gilbert61, Penrose03, hirsch2016}.
However, in many applications it is of interest to study how the network changes over time as the nodes are moving. 
For example sensors can be attached to mobile platforms or carried by firefighters to monitor the environment or the nodes model persons using device-to-device communication with their smartphones, tablets etc.
The properties of such networks change over time and initially uncovered locations can be covered at a later point in time or intruders can be located even if they are not covered by the initial network. Of course, this also means that previously covered regions become uncovered. 
We study the asymptotics of the two random times detection (the time until some target point gets in contact with a node of the network) and coverage (the time until a certain region has been completely discovered by the nodes of the network). Furthermore we investigate when the network has an infinite component. Obviously, in the mobile scenario the network properties not only depend on the initial configuration of the model but also on the mobility behaviour of the devices.
We show how the theory of (marked) Poisson point processes (PPP) can be applied to dynamic telecommunication networks.

The remainder of this paper is structured as follows:  section~\ref{sec:network} introduces the random ad hoc network that  we consider and section \ref{sec:model} provides a detailed description of our mobility model. Section \ref{sec:static detection} deals with the time until a static, uniformly placed node is detected by the network as well as the time until a whole area has been detected by the mobile nodes. Additionally, the detection time of a mobile node is considered. Section \ref{sec:stationary distribution} gives the time stationary distribution of nodes and section \ref{sec:percolation} studies the existence of an infinite component of the network under different assumptions on the network.

\subsection{The random network}\label{sec:network}
\noindent Let $\eta$ be a homogeneous PPP of intensity $ \lambda \in (0,\infty) $ on a convex bounded domain $ D\subset \R^2 $ with  $ \eta(D)=N $, i.e. $ N$ is Poisson-distributed with parameter $\lambda \cdot \text{vol}_2(D), $ where $ \text{vol}_2 $ denotes the two-dimensional Lebesgue volume. The locations of the points of $ \eta $ shall be denoted by $ h^{(1)},\ldots, h^{(N)} $. We now place $ N $ participants $ X^{(1)}, \ldots, X^{(N)} $  at the points of $ \eta $.

Each of the $ N $ participants of our model now independently moves according to the mobility model we will introduce in section~\ref{sec:model}.
These $N$ random walkers built a so-called {\it ad hoc network}, meaning that without any other infrastructure a node of the network can forward an information to all users within distance $ 2r $. The communication allows multihops so that information can be transmitted from a sender to a recipient whenever there is a chain of individuals connecting them such that all links are not larger than $ 2r $. If the nodes of such networks are allowed to move, these networks are also referred to as Mobile Ad hoc NETworks (often called MANETs).

While in applications mobility is a key feature, from a mathematical perspective there are many rigorous results for ad hoc-networks, but only very few considering dynamics.
In \cite{Peres2010} the authors model the movement of the nodes by independent Brownian motions, whereas in \cite{Liu2013} the nodes follow random curves. Both of these works deal with similar properties of the network as those that we discuss in this work. Due to the differences in the movement models diverse results are obtained.

\subsection{The mobility model}\label{sec:model}
The movement model we introduce is a generalization of the well-known \textit{ random waypoint model} (see for example \cite{Roy11}) in which we extent the original model by including so-called 'home sites' to which nodes return. The most intuitive example to think of  
is that of people moving in some domain, returning to their homes now and then.
Though our model is inspired by this application, it also makes sense in other contexts, e.g. one can think of the home sites as base stations storing data sensors collect during their walks. 
We will now give a detailed description of the mobility model and refer to it as the \textbf{sedentary random waypoint model (SRW)}: 

As in the usual random waypoint mobility model let $ D\subset \R^2 $ be a convex bounded domain. 
 From the starting point $W_0\sim \mu_0$, where $ \mu_0 $ is some distribution of starting points on $ D $, one picks a new site $W_{1} $ according to a measure $ \mu $ on $ D $, and a  velocity $ V_{1} $ according to a measure $ \nu $ on $ [v_-,v_+] \subset (0,\infty) $ and moves with constant speed on a straight line towards this site. Continuing this procedure, let $W_n\sim \mu$ and $V_n\sim \nu$ be the random variables denoting the new waypoint velocity after the $n$th change of directions and assume that they are independent of each other as well as  of all other random variables. While in the classical random waypoint model a new waypoint is picked according to  the uniform distribution on the domain $ D $, we allow any waypoint measure $\mu$ which has a positive, continuous Lebesgue density. This enables us to model the popularity of places. We assume that the velocity measure $\nu$ on $ [v_{-},v_{+}]\subseteq (0, \infty) $ has a continuous Lebesgue density, too.
Additionally, we associate a point of the PPP $ \eta $ introduced in the previous section to each of the walkers. The walker returns to this point $h$, which we will refer to as "home", after some (random) time $\alpha$.
The SRW  has been implemented in the BonnMotion software, 
see \cite{AB20}.
We will use the following Markov chain to sample the positions of a walker in the SRW:
\[
 (Y_{n})_{n\in \N} = (W_{n-1},W_{n},V_{n}, R_{n})_{n\in \N} \in D\times D\times [v_{-},v_{+}]\times (0, \infty).
\]
 Here, $ R_{n} $ are continuous random variables on $ (0,\infty)$, which will be introduced below. 
Let  $d(\cdot,\cdot)$ denote the Euclidean distance on $ \R^2 $  and define $ S_n:=\frac{d(W_{n-1},W_n)}{V_n} $ as the time needed for the $n$th stretch of way. We now choose the next waypoint in the following way
 \begin{align*}
 Y_{n+1}=(W_{n},W_{n+1},V_{n+1}, R_{n+1})\ \text{ with }\  
  W_{n+1} 
  \begin{cases}
  =h, \text{ if } R_n-S_n \leq 0\\
  \sim \mu, \text{ if } R_n-S_n >0,
  \end{cases}
 \end{align*}
 $V_{n+1}\sim \nu$ and put  $R_{n+1}=R_n-S_n $ for $ R_n-S_n >0$ 
 resp. $ R_{n+1}= Z_{n+1}+S_{n+1} $ with $ Z_{n+1} $ 
 being an independent random variable distributed according to a measure $\alpha$ on $ (0, \infty)  $ to control the time between two homecomings.
 The idea behind this definition is, that when the walker starts from home an alarm is set and as soon as it rings the walker returns to its home after reaching the next waypoint. $ R_n $ thus captures the remaining time until the next alarm sets off. Here, we allow the alarm time $ \alpha  $ to be either random or deterministic.
Note that since $ Y_{n+1} $ only depends on $ Y_{n} $, $ (Y_{n})_{n\in \N} $ is a Markov chain.

\begin{remark}
As mentioned before, this setup also allows for waypoint measures including 'hotspots', e.g. areas which are visited by a large amount of people frequently, see \cite[B(4)]{HLV06}
 and each walker returns to his/her home  after a random amount of time, see \cite{HMSH05}.\\
 Furthermore the waypoint measure might as well be time-dependent, like in 
 \cite{FlutesCello}, to model different preferences during the day. Here, this means to consider a new measure $ \mu $ -- satisfying the assumptions above -- at each waypoint.
 The SRW follows the ideas of  time-variant and home-cell community-based mobility models, see \cite[VIII, 37]{Roy11} for a survey. 
\end{remark}

\begin{remark}
	All our results also hold in a slightly different model in which we set a new alarm at each waypoint, e.g. $ R_n\sim \alpha $ for all $ n $ and $ W_{n+1}=h$ if $R_n\leq S_n $  and $ W_{n+1}\sim \mu $ otherwise. So at each waypoint, there is a decision based on $\alpha$ whether the next waypoint is the home or a new random waypoint. This model might be useful for example in the context of sensors monitoring a region and which need to be recharged or need to store data (i.e. return to their base station) before taking very long trips.
\end{remark}
In the following sections we study the asymptotic behaviour of the fundamental properties detection, coverage and percolation.

\section{Detection and coverage time}\label{sec:static detection}
 Let $$ T_{n}:=\sum_{j=1}^{n}\frac{d(W_{j-1},W_{j})}{V_j}=\sum_{j=1}^{n}S_{j}$$ denote the time at which the walker reaches waypoint $W _{n} $.
 For $ N(t)= \inf \{n \in\N: T_{n}>t\} $ the position $ X_{t} $ of the walker at time $ t $ is given by
 \begin{align}\label{position}
 X_{t}= W_{N(t)-1}+V_{N(t)} \frac{W_{N(t)}-W_{N(t)-1}}{d(W_{N(t)},W_{N(t)-1})}(t-T_{N(t)-1})
 \end{align}
 and thus we obtain the continuous-time Markov process
 $
 Y_{t}=(X_{t}, W_{N(t)},$ $ V_{N(t)}, R_{N(t)})
 $ 
 on the state space $ \mathcal{D}= D \times D \times [v_{-},v_{+}]\times (0,\infty) $.
 Furthermore, let
 \begin{equation}\label{def_m}
 M(t)=\max\{m \in \mathbb{N}:T_{m} \leq t\} 
 \end{equation}
 denote the number of newly chosen waypoints up to time $ t $. Since the longest possible trip takes time $\frac{tv_-}{\diam(D)}$  we have  $M(t) \geq \frac{t v_-}{\diam(D)} $.

We now add an additional, non-mobile, participant $ w $ in $ D $. We are interested in the time $ T_{Det} $ until $ w $ first gets in contact with one of the other participants of the network, e.g. 
\begin{align*}
T_{Det}&=\inf\bigl\{t\geq\! 0:\! \min_{i=1,\ldots,N}\! d(w, X^{(i)}_{t})\leq 2r\bigr\}=\inf\bigl\{t\geq\! 0: w \in \bigcup^{N}_{i=1}B(X_{t}^{(i)},2r)\bigr\}.
\end{align*}
The problem of detection of a stationary additional participant appears for instance in contexts, where mobile sensors explore a region after natural disasters when the area is unsafe for humans to enter. We get the following result for our movement model.
 \begin{theorem}\label{thm:static detection}
 	For every $ i $, let $ \mu^{(i)} $ be a measure  on a (large) convex domain $ D $ with $  \mu^{(i)}(A)> 0 $  $ \forall A\subset D, A \neq \emptyset  $, let  $ \nu^{(i)} $ be a measure on $ [v^{(i)}_{-},v^{(i)}_{+}]\subseteq (0, \infty) $ and let all nodes $ X^{(i)} $ of the network move according to the SRW with the respective waypoint measures $ \mu^{(i)} $ and velocity measures $ \nu^{(i)} $. Then for any $ w \in D _{-2r}:=\{x\in D: d(x,\partial D)\geq 2r\}$,
 	\[
 	\P(T_{Det}(w)>t):=\P(T_{Det}>t) \leq c_1 \exp(-c_2t),
 	\]
 	where  
 	\begin{align*}
 	c_1&= \frac{\max\{1-e^{\lambda \pi r^2}, e^{\lambda \pi r^2}\}}{1-q},\\
 	 c_2&=-{\frac{v_-}{\diam(D)}} \max\bigl\{\log(1-\min_{1\leq i \leq N}q_{i,1}),\log(1-q)\bigr\}>0
 	\end{align*} 
 	with
 	$
 	q=\max_{i=1,\ldots, N}\mu^{(i)}(B(0,2r))
 	$ and $ q_{i,j}=\P(S_{j}^{(i)} < R_{j}^{(i)}\vert \alpha_0-T_{j-1}>0).$
 \end{theorem}
 
 \begin{proof}
 As we require the point process $ \eta $  of homes to be a homogeneous PPP we can assume without loss of generality that this new participant is located at the origin $0$ and that $ 0\in D_{-2r} $.
Conditioning on the events that the nearest point of our home point process is within distance $ 2r $ of the origin, respectively that all homes have at least distance $ 2r $ from it, we get
\begin{align*}
\P(T_{Det}> t) =&\P\bigl(T_{Det}> t\big|d(0,\eta)\leq 2r\bigr)\cdot \P(d(0,\eta)\leq 2r)
\\&{}+\P\bigl(T_{Det}> t\big|d(0,\eta)> 2r\bigr)\cdot\P\bigl(d(0,\eta)> 2r\bigr),
\end{align*}
where $ d(0, \eta)=\min_{i=1,\ldots, N}d(0,h^{(i)}) $.
As the homes of the walkers are distributed according to a PPP with intensity $ \lambda$, we get
\begin{align*}
\P\bigl(d(0,\eta)\leq 2r\bigr)&=\P\bigl(\eta(B(0,2r))> 0\bigr)
=1-\exp\bigl(-\lambda \text{vol}_2\bigl(B(0,2r)\bigr)\bigr)
\end{align*}
and consequently
$
\P(d(0,\eta)> 2r)=\exp(-4\pi\lambda r^{2} )
$. 
So we  just have to deal with the two conditional probabilities. For the first one we get the following upper bound:
\begin{align*}
\P\bigl(T_{Det}> t\big|d(0,\eta)\leq 2r\bigr)&=\P\bigl(T_{Det}> t\big|\exists i \in \{1,\ldots,N\}: h^{(i)}\in B(0,2r)\bigr)
\\
& \leq \P(\tau{(i)}> t)\leq \max_{i=1,\ldots,N}\P(\tau{(i)}> t),
\end{align*}
where $ \tau{(i)} $ denotes the time between two homecomings of $ X^{(i)} $ and whenever there is more than one home located in $ B(0,2r) $ we choose $$ i=\argmin_{j: d(0, h^{(j)})\leq 2r} \P(\tau(j)>t).$$
Define the probability that walker $ i $ does not come home after the $ j $th trip, given he started at his home at time 0 and has not come home ever since, as $ q_{i,j} $, i.e. $ q_{i,j}:= \P(S_{j}^{(i)} < R_{j}^{(i)}\vert \alpha_0-T_{j-1}>0). $   
We now get
\begin{equation*}
\P(\tau{(i)}> t)
=\prod_{j=1}^{M(t)}(1-q_{i,j})\leq (1-q_{i,1})^{M(t)} ,
\end{equation*}
recalling that by definition \eqref{def_m} $ M(t) $ denotes the number of newly chosen waypoints up to time $ t $. 
With $ q^\ast=\min_{1\leq i \leq N}q_{i,1}  >0$ 
we get
\begin{equation*}
\max_{i=1,\ldots,N}\P(\tau{(i)}> t)=(1-q^\ast)^{M(t)}\leq e^{-c t}, \text{ for } c = -\log((1-q^\ast)^{\frac{v_-}{\diam(D)}})>0.
\end{equation*}

\noindent Defining $ q:=  \max_{i=1,\ldots, N}\mu^{(i)}(B(0,2r))
>0$  
yields
\[
\P\bigl(T_{Det}> t\big|d(0,\eta)> 2r\bigr) \leq (1- q)^{M(t)-1},
\]
with $ M(t) \geq \frac{t v_-}{\diam(D)}$. For $c_2=-{\frac{v_-}{\diam(D)}} \max\{\log(1-q^\ast),\log(1-q)\}>0 $ we obtain the result.
\end{proof}


The next result deals with the coverage time $ T_{Cov}(A) $ needed until an area $ A\subseteq D  $ has been completely discovered by the participants of the network.
\begin{corollary}
	With the same assumptions as in Theorem \ref{thm:static detection}, for all $ A\subset D_{-2r} $, there exists a positive constant $ C_{A} $ depending on $ r $ such that
	\[
	\P(T_{Cov}(A)>t)\leq C_A\exp(-c_2 t) ,\label{coveragetime}
	\]where $ c_2 $ is the same constant as in Theorem~\ref{thm:static detection}.
\end{corollary}

\begin{proof}
 Let $ 0<\varepsilon<r $ be fixed and define $M_{\varepsilon}(A)$
 as the number of balls with radius $\varepsilon$ needed to cover  $ A $.
 Furthermore let $ Z_{t} $ denote the number of $ \varepsilon $-balls in $A$ that have not been discovered by the walkers up to time $ t $. Using the Markov-inequality we get
\[
\P(T_{Cov}(A)>t)=\P(Z_{t}\geq 1)\leq \E[Z_{t}]
\]
and 
$
\E[Z_{t}]\leq M_{\varepsilon}(A) \cdot\P(\text{a given ball }B(x, \varepsilon)\text{ is not covered by time } t).
$
A given ball $ B(x, \varepsilon) $ is discovered as soon as one of the walkers enters $ B(x,r-\varepsilon) $. For this event we can apply the previous theorem, thus
\[
\P(T_{cov}(A)>t) \leq c_1M_{\epsilon}(A)\exp(-c_2  t)= C_A\exp(-c_2 t),
\]
for $ t >1 $ and where $ C_{A} $ depends on $ r , \varepsilon $ and $ A$.
\end{proof}

\vspace{0.5cm}
 In the following we deal with the case of an additional node $w$, placed at its home $ h_{w}$ in $ D $, again, but this time we assume that $ w $ moves according to the SRW given by the measures $ \mu^{(w)} $ and $ \nu^{(w)} $. 
 Again we are interested in the time it takes until $ w $ can communicate with one of the other participants for the first time.
 \begin{theorem}
 	With the same assumptions as in Theorem \ref{thm:static detection} we get the following result for the mobile case:
 We have
 	\[
 	\P(T_{Det, mobile}>t)\leq \exp(-ct), 
 	\]
 	where $ c=\frac{\lambda v_- \cdot\min \P(X^{(i)}\in B(h^{(i)},r))}{\diam(D)}. $
 \end{theorem}

 Before we prove the theorem, let us point out, that in the case of the SRW the distribution of the detection time has thinner tails than in the Mobile Geoemtric Graph model where points are moving according to a Brownian motion and the exponent grows with $t/\log(t)$, see \cite{KKP03, Peres2010}.

 \begin{proof}
 	 We write $ p^{(i)}=\P\bigl(X^{(i)}\in B(h^{(i)},r)\bigr) $ for the proportion of time node $ i $ spends within distance $ r $ of its home. As each participant of the network  returns home infinitely often with probability 1 as $ t \rightarrow \infty $, we have $ p^{(i)}> 0 \ \forall i=1,\ldots, N $ and since $ N $ is a.s. finite there exists an $ \varepsilon > 0 $ such that $ p^{(i)}\geq \varepsilon \ \forall i=1, \ldots , N $.
 We now sample the location of the additional walker $ w $ at times $ t_{1},\ldots, t_{m} $, where $ t_{i}-t_{i-1}> \frac{\diam(D)}{v_-} $, so that each walker has changed directions at least once in the time interval  $ [t_{i-1},t_{i}) $. Hence, the thinnings we will use later are independent for all $ i $.
 At each time step $ t_{i} $ we now  add marks $ Y_{t_{i}}^{(j)} $ to the nodes $ h^{(j)} $ of the PPP of homes according to the following rule:
 \begin{align*}
 Y_{t_{i}}^{(j)}= \begin{cases}
 1, &\text{ if } \exists k \in \{1,\ldots, N\}: d(h^{(j)}, X_{t_{i}}^{(k)})\leq r\\
 0, &\text{else}.
 \end{cases}
 \end{align*}
 We have $\P(Y_{t_{i}}^{(j)}=1):= p^{(j)}_{\ast}> p^{(j)}\geq \varepsilon $. We now consider the thinned process $ \eta_{i} $ at time $ t_{i } $ which only consists of those $ h^{(j)} $ for which $ Y_{t_{i}}^{(j)}=1 $. $ \eta_i $ is again a PPP, though  not necessarily homogeneous. 
 Each of the point processes $ \eta_i $ clearly dominates the stationary PPP $ \eta^{\ast} $ with intensity $ \lambda \cdot \min_{j=1,\ldots,N}{p^{(j)}} $ in terms of the number of nodes.
 Recalling \eqref{def_m} this yields
 \begin{align*}
 &\P(T_{Det}> t)\leq \P(T_{Det}> t_{M(t)})\leq \P\bigl( \eta_{i}(B(w_{t_{i}},r))=0 \ \ \forall i=1,\ldots, M(t)\bigr)\\
 &=\E\!\left[\E\left[\mathds{1}\bigl\{\eta_i\bigl(B(w_{t_i},r)\bigr)=0\ \forall i=1,\ldots, M(t)\bigr\}\big|w_t\right]\right]\\
 &
 =\E\!\Biggl[\prod_{i=1}^{M(t)}  \P\bigl( \eta_{i}(B(w_{t_{i}},r))=0\big|w_{t_i}\bigr)\!\Biggr]
 \leq \prod_{i=1}^{M(t)} \P\bigl( \eta^{\ast}(B(0,r))=0\bigr)\\
 &= \P\bigl( \eta^{\ast}(B(0,r))=0\bigr)^{M(t)}
 \leq \exp\Big(-t v_-\cdot \frac{\lambda \min p^{(j)}}{\diam(D)}\Big) = \exp\Big(-c \cdot t \Big),
 \end{align*}
 where $ w_t=(w_{t_1},\ldots,w_{t_{M(t)}}), \eta=(\eta_1,\ldots,\eta_{M(t)}) $ and we used that for every $ i $ the point process $ \eta_i $  stochastically dominates the homogeneous PPP $ \eta^{\ast} $ with parameter $ \lambda \cdot \min_{j=1,\ldots,N}{p^{(j)}} $ as well as the fact that   $ M(t)\geq \frac{t v_-}{\diam(D)} $.
  \end{proof}
  
\section{Time-stationary distribution}\label{sec:stationary distribution}
  For the last two sections we will have to make some additional assumptions on the waypoint measures as well as on the distribution of the alarm time $ \alpha. $ Namely,the waypoint measures $ \mu^{(i)} $ on $ D $ are no longer allowed to change in time and should be
  \begin{itemize}
  \item[(A1)] \textbf{centered}, i.e. for all  $ i $ and all measurable $ A\subset D $,
  $$
  \mu^{(i)}(A)=\P(W_n^{(i)}\in A )=\mu(A-h^{(i)}),
  $$
  here $ A-h^{(i)}$ denotes the translated set $ A-h^{(i)}=\{x-h^{(i)}: x \in A \}. $ 
  \item[(A2)] \textbf{exponentially decaying}: for each $ i $ there exists some $ \beta >1 $ satisfying \\
  $\P (d(h^{(i)},W_j^{(i)})>s)\asymp C ~ s^{-\beta}$
  for some constant $ C $ and $ s$ large enough.
  \item[(A$\alpha$)] Furthermore we assume that $ Z_n^{(i)}\sim \alpha_n $ for all $ i $ and some distribution $ \alpha_n $ on $ (0,\infty) $.
  \end{itemize}
  \begin{theorem}
  	For the SRW satisfying assumptions \textnormal{(A1), (A2)} and \textnormal{(A$\alpha$)}, there exists a unique time-stationary distribution $ \pi $ of  $ Y_n=(W_{n-1},W_n,V_n, \alpha_n) $.
  \end{theorem}
  \begin{proof}
  We  apply Theorem 6 in \cite{LV06}. We first of all show that our model fits into the definition of a random trip model as described by the five items given in \cite[Section IIa]{LV06}. By the definition of our model we have 1) domain: $ D\subset \R^2 $, 2) phase: there are no different phases in our model, so we can for example put $ \mathcal{I}=\{\text{walking}\} $ as the only phase, 3) path: all paths in the SRW model are continuous mappings from $ [0,1] $ to $ D $ given by
  \[
  p_n(u)=W_n+(W_{n-1}-W_n)(1-u)
  \]
  which have continuous derivatives except at change points.
  4) Trip: the position of $ X(t) $ of the mobile node at time $ t $ is defined 
  \[
  X(t)=p_{M(t)} \left(\frac{t-T_{M(t)}}{S_{M(t)}}\right)
  \]
  where $ S_{M(t)}=\frac{d(W_{M(t)-1},W_{M(t)})}{V_{M(t)}} $ denotes the length of the trip the walker is taking at time $  t$, cf. \eqref{position}. 
  5) Default Initialization Rule: at time $ t=0 $ we place the walkers at their homes and draw the first waypoint $ W_1 $ according to $ \mu $, a velocity according to $ \nu $ and an alarm time according to $ \alpha_{1} $.\\
   Now, in order to obtain the time stationary distribution we have to check conditions (H1)-(H4) in \cite[Section IIb]{LV06}. (H1) and (H4) are clear by the assumptions made on our model, (H2) and (H3) work analogously to the case of the classical random waypoint model (cf. \cite{DFK16} section 3.1). Thus it only remains to show that $ \E\left[S_0\right]< \infty $.
  By our assumptions on $ \mu $ we get that $ \P(d(0,W_i)>t)\leq C t^{-\beta} $ for some $ \beta > 1 $,  and  $ t $ large enough (say $ t \geq T $ for some $ T < \infty $). Without loss of generality, we assume that the home of our walker is located at the origin, hence, $ W_0=0 $. Now
  \begin{align*}
  	\P(S_0>t)&=\P\bigl(d(0,W_1)>tV_0\bigr)\leq C
\cdot  \E\!\left[\left(t V_0\right)^{- \beta}\right]\leq 
C \int_{v_-}^{\infty}\biggl(tv\biggr)^{-\beta }f_{\nu}(v) \d v,
  \end{align*}
as we assumed $ \nu $ to be absolutely continuous.
  Thus, by Fubini's theorem $\E\left[S_0\right]=\int_{0}^{\infty}\P(S_0>t)\d t$ is bounded by
  \begin{align*}
  	\E\left[S_0\right]&\leq T+\frac{C}{v_-^{\beta}}\int_{T}^{\infty}\int^{\infty}_{v_-} t^{-\beta }f_{\nu}(v) \d v \d t\leq T+\frac{C}{(\beta-1)\,v_-^{2\beta-1}}\int_{v_-}^{\infty}f_{\nu}(v) \d v\\
  	&=T+\frac{C}{(\beta-1)\,v_-^{2\beta-1}}< \infty.
  \end{align*}
  By Theorem 6 in \cite{LV06} we get that $ \Phi(t):=(Y(t),S(t),\frac{(t-T_{N(t)-1})}{S(t)}) $, where $ S(t)=\frac{\lvert W_{N(t)-1}-W_{N(t)}\rvert}{V_{N(t)}} $, possesses unique time-stationary distribution $ \pi $ given by
  \begin{equation}
  	\pi(B)=\frac{\E^0\left[\int_{0}^{T_1}\mathds{1}\{\Phi(s)\in B\}\d s\right]}{\E^{0}\left[S_0\right]},\label{pi}
  \end{equation}
  where $ \E^0 $ denotes the "Palm expectation", which can be interpreted as the expectation conditional on the event that a transition occurs at time $ 0 $.
  \end{proof}

 \section{Percolation} \label{sec:percolation}
\subsection{Percolation under assumptions (A1), (A2) and (A$ \alpha) $}
In this section we will deal with \textit{percolation}, i.e. with the existence of an infinite large component of the network. In order to speak of infinite components we necessarily have to drop the assumption that $ D $ is a bounded domain and instead assume that $ D=\R^2 $.
 As we will see in Lemma~\ref{le:InfiniteComponent}, assumptions (A1),(A2) and (A$ \alpha $) ensure that a unique infinite component exists with probability one, if and only if the underlying PPP of homes is in the super-critical regime., i.e.  $ \lambda > \lambda_c $ so that there exists a unique largest cluster of homes with probability one.
  \begin{lemma}\label{le:InfiniteComponent}
  	Under the assumptions on $ \mu $, $ \nu $ and $ \alpha $ given above, there exists a unique infinite component of walkers if and only if the homogeneous PPP of homes is of  intensity $ \lambda>\lambda_c $.
  \end{lemma}
  \begin{proof}
  		Let $ \eta_h $ denote the homogeneous PPP of homes of intensity $ \lambda $  on $ D=\mathbb{R}^2 $.
  		Under all our assumptions the PPP $ (\eta_s^W) $ of walkers at time $ s $ can be seen as a random translation of the home process $ \eta_h $, such that each home is displaced by a random vector with probability density function $p(x,\cdot) $. Notice that when looking from the homes of the walkers all of them possess the same waypoint measure $ \mu $, thus the probability kernel $ p $ is the same for all walkers. As each walker moves independently from all others also the displacements happen independently. Following the displacement theorem in  \cite[p.61]{kingman1992poisson} the new displaced point process $ \eta^{h,d}=\eta^W $ is again a PPP with intensity function 
  		\begin{equation*}
  		\lambda_d(y)=\int_{\R^d}\lambda p(x,y)dy = \lambda.
  		\end{equation*}
  		Thus the existence of an infinite component of walkers is equivalent to the existence of an infinite component of homes.
  \end{proof}

\subsection{Percolation -- Towards a Phase Transition}

In this section we follow the idea that on the one hand -- in the case that walkers do not return to their homes very often
-- the Sedentary Random Waypoint model shows the same behaviour as the classical Random Waypoint model, and on the other hand -- for  the walkers favouring short trips around their homes-- the connectivity behaviour rather resembles that of the random geometric graph.
For this section, we specify the model further in order to have a parameter classifying these two extremes.
Let $\eta_h$ be a homogeneous Poisson point process in the window $W=[0,a]^2\subset\R^2$. 
%
The decision where to go next will be made according to the following rule: with probability $ p $ the walker chooses to take a long trip, more precisely the next waypoint will be uniformly distributed on $W\backslash \{B_R(h^{(i)})\}$. With probability $ (1-p) $ the walker stays close to home and the waypoints are uniformly chosen in a ball of radius $ R $ around the home. Fix the waypoint radius $R$.
For $ p=0 $ we are then close to the random geometric graph (RGG), whereas for $ p=1 $ we almost get the RWP model, so that our model interpolates between the two.

For the RGG it is well known, e.g. \cite{LastPenrose}, that depending on the communication radius $r$ there exists a critical value $\lambda_c(r)$ such that the model is in the supercritical phase, i.e. percolation occurs a.s. for all $\lambda>\lambda_c(r)$.

At time $ s $, we define $ \eta_s $ as the point process of walkers at time $ s $. We then delete all walkers that are not close to their home at time $s$ and time $s-1$ . We denote the resulting, thinned process of intensity $ (1-p)^2\lambda $ by $ \eta_s^{(t)} $.
Comparing this point process to the RGG with parameter $r-\frac{R}{2}$, we see that the random graph is in the supercritical regime if
\[
(1-p)^2\lambda > \lambda_c\Bigl(r-\frac{R}{2}\Bigr).
\]
This proves the existence of an crossing component for $p\to 0$ for $\lambda$ satisfying
$\lambda>\lambda_c\Bigl(r-\frac{R}{2}\Bigr)$.

For the other extreme, $p\to 1$, we fix a window $W_a=[0,a]^2$ and study the limit $a\to\infty$. Following \cite{Bettstetter} the (approximate) spatial distribution of points is given by
\[
f(x,y)\approx \frac{36}{ a^6} ~\Big(x^2-ax\Big)\Big(y^2-ay\Big).
\]
We introduce the measure $ \mathbb{Q}_a:\mathcal{B}(\R^2) \rightarrow \R$ given by
\[
\mathbb{Q}_a(B)=\int_B \mathbf{1}\{(x,y)\in W_a\} f(x,y)\, d(x,y).
\]
Denote by
$S_\delta= [\delta, a] \times [0, a]$ a stripe of width $a- \delta $ on the right border of $ W_a $. Note that  the density of the spatial distribution is decreasing in $ x $  for $ x \geq \frac{a}{2} $ and that, for fixed $ x  $, $ f(x,y) $ is maximal for $ y=\frac{a}{2}$. With these considerations we see that for all $ (x,y) \in S_\delta $ the density $ f(x,y) $ is uppper bounded by $ f(\delta, \frac{a}{2}) $. This holds especially for $ \delta= \sqrt{a} .$ Hence, the density of walkers in
$ [\sqrt{a},a] \times [0,a] $ can be bounded by 
\[
f(\sqrt{a},\frac{a}{2})\approx 9\frac{a^3(\sqrt{a}-1)}{a^6}=9 \frac{\sqrt{a}-1}{a^3}.
\]
Remember, that the underlying point process $ \eta_h $ is assumed to be a homogeneous PPP of intensity $ \lambda $, i.e.~the expected number $ \E\left[N_{W_a}\right] $ of points in $ W_a $ is given by $ \lambda a^2 $. If we now consider the number  $ N_{\delta} $ of points in $ S_{\delta} $ we obtain
\begin{align*}
\E\left[N_{\delta}\right]=\E\left[\E\left[N_\delta\vert N_{W_a}\right]\right]&=\E\left[\sum_{k=1}^{N_{W_a}}\mathbf{1}\{X_k\in S_\delta\}\right]=\E\left[N_{W_a}\right]\P(X_k\in S_\delta)\\
&=\lambda a^2\mathbb{Q}(S_\delta)=  \lambda a^2 \int_{S_\delta} f(x,y)\, d(x,y).
\end{align*}
Hence the point process of walkers in $ S_{\sqrt{a}} $ is stochastically dominated by a homogeneous PPP of intensity 
\[
\E\left[N_{\sqrt{a}}\right]/\text{vol}_2(N_{\sqrt{a}})\leq 9\lambda \int_{\sqrt{a}}^{a}\int_{0}^{a}\frac{1}{\sqrt{a}} d(x,y)/\text{vol}_2(N_{\sqrt{a}}) = \frac{9\lambda}{\sqrt{a}}.
\]
Thus, for $a\rightarrow \infty$, there are a.s. no points on the borders of the observation window, so that no percolation occurs.\\
\\


Is one on the other hand not interested in covering the regions but in reachability among each other, then one can analogously bound the density in the square $\bigl[\frac{a}{2}-\sqrt{a},\frac{a}{2}+\sqrt{2}\bigr]^2$ from below by
\[
 f(x,y)\geq f\bigl(\frac{a}{2}-\sqrt{a},\frac{a}{2}-\sqrt{2}\bigr)
 \approx \frac{36}{a^6} \bigl(\frac{a^4}{16} - \frac{a^3}{2}+a^2\bigr).
\]
So the process stochastically dominates the PPP with intensity
\[
 \lambda \cdot a^2 \cdot \frac{36}{a^6} \bigl(\frac{a^4}{16} - \frac{a^3}{2}+a^2\bigr)
 = \frac{9}{4} \lambda - \frac{18 \lambda}{a} + \frac{\lambda}{a^2}.
\]
This shows, that if $\frac{9}{4} \lambda > \lambda_c(r)$ and $a$ big enough, an infinite component exists in the center.\\
\\
\\
 \bibliographystyle{alpha} 
 \bibliography{MobileNetworks}
 
%
%
%


\end{document}